\documentclass[11pt, reqno]{amsproc}

\usepackage[margin=1.0in]{geometry}
\usepackage[english]{babel}
\usepackage[utf8]{inputenc}
\usepackage{amsmath,amssymb,amsthm, comment, mathtools}
\usepackage{hyperref,mathrsfs, xcolor}
\usepackage{tcolorbox}

\newcommand{\R}{\mathbb{R}}

\newcommand{\pa}{\partial}
\newcommand{\ve}{\varepsilon}
\newcommand{\res}{\rho^*} 

\newcommand{\eee}{equation}
\newcommand{\be}{\begin{\eee}}
\newcommand{\ee}{\end{\eee}} 
\newcommand{\rmd}{{\rm d}}

\let\div\relax 

\DeclareMathOperator{\curl}{curl}
\DeclareMathOperator{\div}{div}

\numberwithin{equation}{section}

\newtheorem{theorem}{Theorem}
\numberwithin{theorem}{section}
\newtheorem{prop}{Proposition}
\numberwithin{prop}{section}
\newtheorem{cor}{Corollary}
\numberwithin{cor}{section}
\newtheorem{lemma}{Lemma}
\numberwithin{lemma}{section}
\theoremstyle{remark}
\newtheorem*{remark}{Remark}

\theoremstyle{definition}
\newtheorem{defi}{Definition}

\usepackage{tikz}
\usetikzlibrary{patterns}
\usetikzlibrary{fpu}
\usetikzlibrary{plotmarks}
\usetikzlibrary{decorations.markings}
\usepackage{pgfplots}
\usetikzlibrary{intersections, pgfplots.fillbetween}

\begin{document}

\title{On the distribution of heat in fibered magnetic fields}

\author{Theodore D. Drivas}
\address{Department of Mathematics, Stony Brook University,
Stony Brook, NY, 11794}
\email{tdrivas@math.stonybrook.edu}
\author{Daniel Ginsberg}
\address{Department of Mathematics, Princeton University, Princeton, NJ 08544}
\email{ dg42@princeton.edu}
\author{Hezekiah Grayer II}
\address{Program in Applied and Computational Mathematics, Princeton University, Princeton, NJ 08544}
\email{hgrayer@math.princeton.edu}
\maketitle

\begin{abstract}
We study the equilibrium temperature distribution in a model for strongly magnetized plasmas in dimension two and higher.  Provided the magnetic field is sufficiently structured (integrable in the sense that it is fibered by co-dimension one invariant tori, on most of which the field lines ergodically wander) and the effective thermal diffusivity transverse to the tori is small, it is proved   that the temperature distribution is well approximated by a function that only varies across the invariant surfaces.
The same result holds for  ``nearly integrable" magnetic fields up to a  ``critical" size.  In this case, a volume of non-integrability is defined in terms of the temperature defect distribution and related the non-integrable structure of the magnetic field, confirming a physical conjecture of Paul--Hudson--Helander \cite{PHH21}.  Our proof crucially uses a certain quantitative ergodicity condition for the magnetic field lines on full measure set of invariant tori, which is automatic in two dimensions for magnetic fields without null points and, in higher dimensions, is guaranteed by a Diophantine condition on the rotational transform of the magnetic field.
\end{abstract}

\section{Introduction}
The heat conduction in strongly magnetized plasmas is influenced locally by the direction of
the magnetic field $B : \R^{d} \to \R^{d}.$ Braginskii \cite{Braginskii}  (see also \cite{HB,HelanderSigmar}) derived an effective anisotropic diffusion equation for the temperature $T$ in such an environment which, in steady state and free of heat sources, reads
	\begin{alignat}{2}
	 \div (b \nabla_b T + \ve \nabla_b^\perp T) &=0
	 \label{temp0}
	\end{alignat}
	where, assuming the magnetic field has no null points $|B|\neq 0$, we introduced
	\begin{equation}
b = \tfrac{B}{|B|}\qquad  \nabla_b =b\cdot \nabla,
 \qquad
 \nabla^\perp_{b} = \nabla - b\nabla_{b}.
 \label{}
\end{equation}
This equilibrium equation 
captures macroscopically the phenomenon that charged particle dynamics strongly influenced by $B$ favors collisions aligned with $b$.
In \eqref{temp0}, the parameter $\ve>0$ represents the ratio  $\kappa_{\perp} /\kappa_\|$ of the transverse diffusion coefficient to the longitudinal.   In general it is a scalar  function of local density and field magnitude $|B|$, however its magnitude is small in many applications of interest where $|B|$ is large.  In our work, we treat $\ve$ as a constant and study the limit $\ve\to 0$.

  For arbitrary $B$, it is not immediate what emerges in the limit
	$\ve \to 0$ of \eqref{temp0}, given some fixed boundary conditions. We focus on toroidal ``Arnold fibered''
fields $B$. These are solenoidal vector fields  $B$ having the property that there is a smooth
function $\psi: D \to \mathbb{R}$ defined in a bounded region $D \subset \mathbb{R}^d$ with $|\nabla \psi|\not=0$ in $D$,
 whose level sets $S_\psi$ are $(d-1)$--dimensional tori such that $\psi$ is a first integral
 \begin{equation}
  B\cdot \nabla \psi = 0.
  \label{fibered}
 \end{equation}
 We shall term these fields \emph{(toroidally) fibered}. 
	In two dimensions,
	if $B$ is divergence-free and sufficiently smooth, then $B = \nabla^\perp A$ for a  ``streamfunction'' $A: D\to \mathbb{R}$ where $\nabla^\perp = (-\pa_y, \pa_x)$. If
	$B$ has no nulls, then $|\nabla A| > 0$, so any
	non-vanishing divergence-free field in two dimensions is fibered by its
	streamfunction, e.g. $\psi=A$.  See Figure \ref{fig:chan} (a).  In three dimensions, its straightforward to write down explicit fibered fields, see \eqref{Bformmix} and Figure \ref{fig:chan} (b).  Moreover, as we will later discuss, non-degenerate  magnetohydrostatic equilbria have this property.
	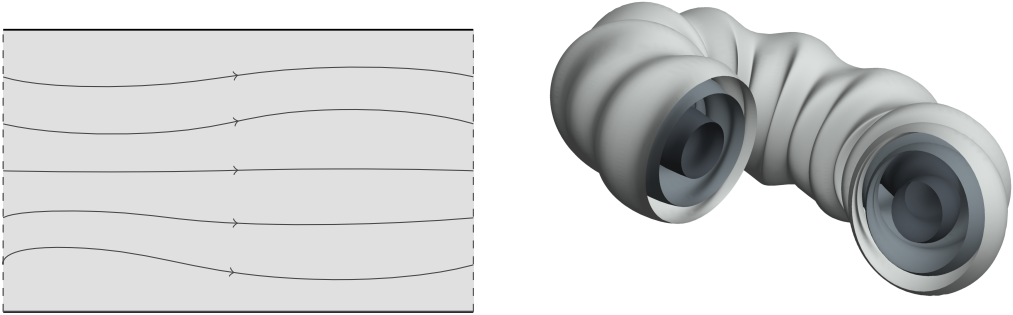
\begin{figure}[h!]
\centering
\begin{tikzpicture}[scale=1.25, every node/.style={transform shape}]

		 \draw [name path=A,thick] (5.5,0)--(10.5,0);
		  \draw [name path=B,thick] (5.5,3)--(10.5,3);

	\draw [thin,   decoration={markings, mark = at position -0.5 with {\arrow{>}}},
        postaction={decorate}]  plot [smooth, tension=1]  coordinates {(5.5,2.5) (6.8,2.4) (9,2.6)   (10.5,2.5)};
	\draw [thin,   decoration={markings, mark = at position -0.5 with {\arrow{>}}},
        postaction={decorate}]  plot [smooth, tension=1]  coordinates {(5.5,2) (7,1.9) (9,2.15)   (10.5,2)};
	\draw [thin,   decoration={markings, mark = at position -0.5 with {\arrow{>}}},
        postaction={decorate}]  plot [smooth, tension=1]  coordinates {(5.5,1.5)  (7,1.49) (9,1.52)  (10.5,1.5)};
	\draw [thin,   decoration={markings, mark = at position -0.5 with {\arrow{>}}},
        postaction={decorate}] plot [smooth, tension=1]  coordinates {(5.5,1) (6.4,1.07) (8.4,0.93) (10.5,1)};
	\draw [thin,   decoration={markings, mark = at position -0.5 with {\arrow{>}}},
        postaction={decorate}]  plot [smooth, tension=1]  coordinates {(5.5,0.5) (6.3,0.68) (8.7,0.35) (10.5,0.5)};

        \draw [dashed] (5.5,0)--(5.5,3);
        \draw [dashed] (10.5,0)--(10.5,3);
\tikzfillbetween[of=A and B]{black!30!white, opacity=0.4};
\end{tikzpicture}
  \qquad  \includegraphics[width=2.5in]{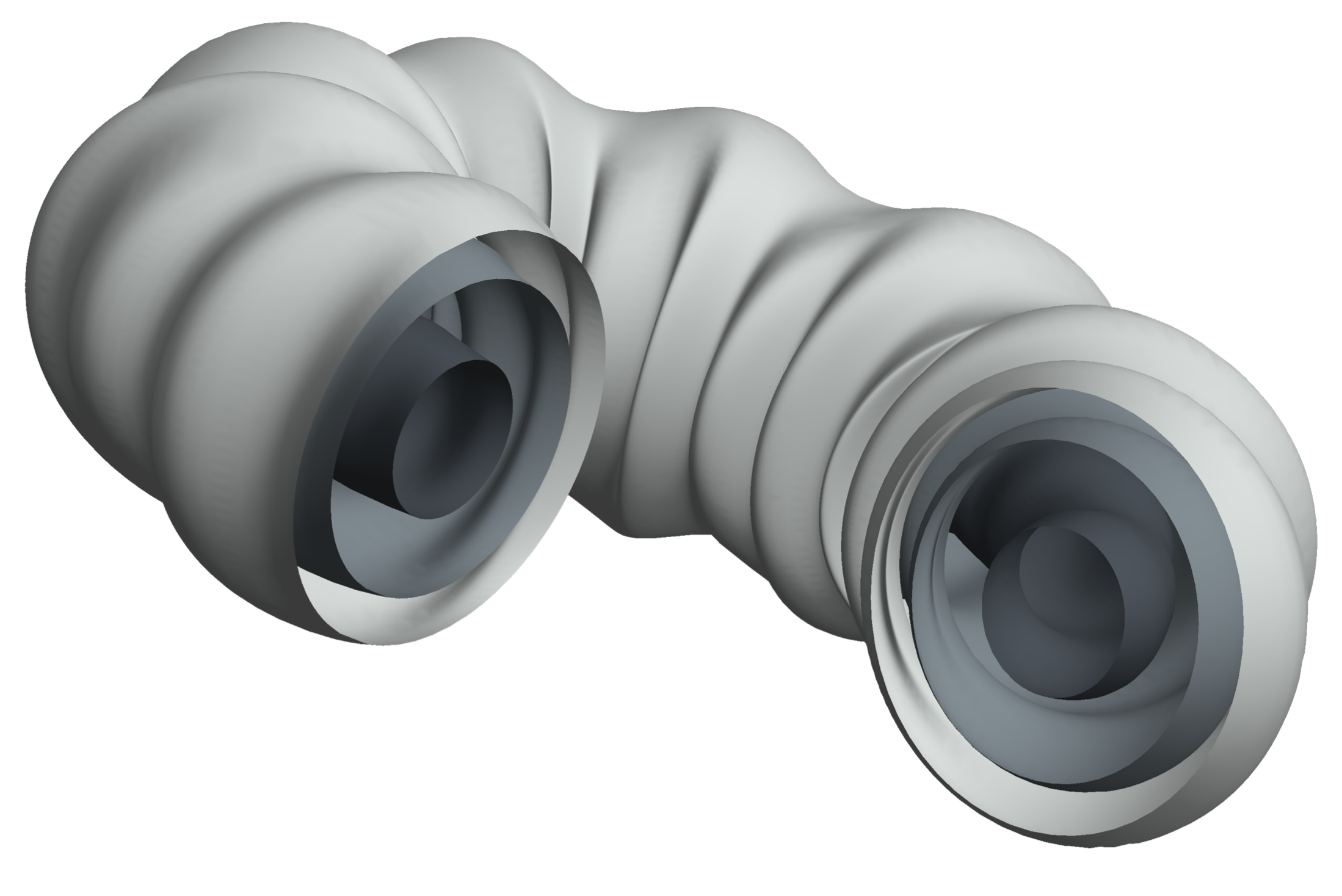}
  \caption{Examples of fibered magnetic fields. Left: a 2d magnetic field without null points on a (topologically) annular domain -- the periodic channel.  Integral curves are levels of the streamfunction.   Right: a 3d toroidal magnetic field; depicted in grayscale are distinct level surfaces of the first integral, $\psi$, the flux function.}
  \label{fig:chan}
\end{figure}

	The temperature equation \eqref{temp0} is to be solved in a  toroidal shell  $D$ with boundaries $S_{\pm}$ that are level sets of the first integral $\psi$.  Call $ \psi_- :=\inf_D \psi$ and $ \psi_+ := \sup_D \psi$. Since $\psi$ is non-degenerate by assumption, $S_{\pm}$ are the levels corresponding to the values $\psi_\pm$. To complete the problem, we  impose Dirichlet boundary conditions
for the temperature field $T:D\to
\R$ on these surfaces. Overall, we consider the system
	\begin{equation}
	\label{ADE}
	\begin{aligned}
	 \div (b \nabla_b T + \ve \nabla_b^\perp T) &=0 &&\qquad \text{ in } D,\\
	 T &=T_{\pm}, &&\qquad \text{ on } S_{\pm},
	 \end{aligned}
	\end{equation}
  for constants $T_{-},T_+$.  The \ref{ADE} system is used in practice as an efficient method to visualize the flux surfaces of the magnetic field \cite{HB,PHH21A,PHH21}.

To state our main result concerning the convergence of $T_{\ve}$, we use some notions from mixing to characterize the behavior of $B$ via its trajectories on the flux surfaces $S_\psi$. Denoting  $I = \lbrack \psi_-, \psi_+\rbrack$, we say that $S_\psi$ is an ``ergodic'' surface if $\psi$ is in the set
\begin{equation}
 E(\gamma,M) := \left\{ \psi \in I : \| u\|_{\dot{H}^{-\gamma}(S_{\psi})}
 \leq M \|\nabla_B u\|_{L^2(S_{\psi})}, \text{ for all } u \in H^1(S_{\psi})\right\} 
 \label{mixingset}
\end{equation}
for some nonnegative $\gamma$ and $M$, where $\dot{H}^{\gamma}(S_{\psi})$ denotes the homogeneous Sobolev space of index $\gamma$ on $S_\psi$.
The sets $E(\gamma, M)$ of ergodic values may be empty or may have full measure, depending on
$B$. The definition of these sets is motivated by a
Diophantine condition, see \eqref{diophantine}.
We then define the collection
\begin{equation}
	N(\gamma, M) = I \setminus E(\gamma, M),
 \label{nonmixingset}
\end{equation}
of ``non-ergodic'' values of $\psi$. Note that if $M > M'$ then $N(\gamma, M) \subseteq N(\gamma, M')$.
\begin{defi}
	\label{ergcond}
	We say that $B$ satisfies the ``ergodicity condition'' if, with
	$N(\gamma, M)$ defined as in \eqref{nonmixingset},
	for some
	$c, \gamma > 0$, we have
	\begin{equation}
	\lim_{M \to \infty} M^c  \mu(N(\gamma, M)) = 0,
	\label{measureergo}
	\end{equation}
	where $\mu$ denotes the one-dimensional Lebesgue measure.
\end{defi}

Our main result below roughly states that, provided $B$ is ergodic on almost all of the surfaces $S_\psi$ such that the ergodicity condition holds, the temperatures profiles $T_\ve$ indeed converge (in $H^1(D)$) to the effective temperature $T_0$. A consequence of our theorem is that the limiting temperature profile $T_0$ itself fibers $B$.  This fact partially motivated the work of Paul--Hudson--Helander \cite{PHH21}.
\begin{theorem}
	\label{mainmixthm}
	Let $d \geq 2$ and let $B$ be toroidally fibered by $\psi$, and let  $D$ be the region bounded by two level sets $S_{\pm}$. For $\ve > 0$, let  $T_\ve:D\to \mathbb{R}$ be the solution of system (\ref{ADE}) for constants $T_{-}$ and $T_+$. If the ergodicity condition from Definition \ref{ergcond} holds,   then 	
\begin{equation}
T_\ve \to T_0:= \Theta(\psi) \qquad \text{ in } H^1(D) 
\end{equation}
where $\Theta(\psi)$ is the solution of the one-dimensional boundary-value problem on $\psi \in \lbrack\psi_- , \psi_+\rbrack$:
	\begin{align}
	 \frac{\rmd}{\rmd\psi}\left( \frac{\rmd \Theta}{\rmd\psi}\int_{S_\psi} |\nabla \psi|\, \rmd \mathscr{H}^{(d-1)} \right) &=0, \qquad \Theta(\psi_{\pm}) = T_{\pm}\label{effectivebc}.
	\end{align}
In fact,  there is $C:= C(D,B)>0$ such that
\be
\label{convrate}
\| T_\ve-T_0\|_{H^1(D)} \leq C \ve^{\frac{c}{2+c}},
\ee
where $c$ is the largest so that there is a $\gamma>0$ making condition \eqref{measureergo}  of  Definition \ref{ergcond} hold.
\end{theorem}

The proof is found in \S \ref{secmainmixthm}. Briefly, if each integral curve of $B|_{S_\psi}$ covers $S_\psi$
densely for some $\psi$, (that is, if $S_\psi$ is an ``irrational torus''),
then
$B$ ``nearly'' spans the tangent space at each point. On such a torus,
one encounters a small divisors problem; the operator $\nabla_B$ is bounded below
on $S_\psi$ but the lower bound may be arbitrarily small.
However, for $\psi \in E(\gamma, M$),
this lower bound cannot be less than $1 /M$.
On the complement $N(\gamma, M)$, the operator $\nabla_B$
is not bounded below,
 but the ergodicity condition
\eqref{measureergo} ensures that the measures of the sets
$N(\gamma, M)$ go to zero as $M$ increases.  The net result is one of homogenization to a one-dimensional limit profile adapted to the geometry of the invariant tori that satisfies an effective diffusion equation.  See \S \ref{secmainmixthm} for further discussion.

In the upcoming Corollaries \ref{2dcor}, \ref{3dcor},
we show that this condition holds for a large family of physically-relevant
vector fields $B$. Whenever $d = 2$, the sets $N(\gamma, M)$ are empty for large enough $M$; that is, every surface $S_\psi$ is ergodic in this setting (in three and higher-dimensions, the ergodicity condition need not be true in general).  Thus $c$ in bound \eqref{convrate} may be taken to $\infty$ for any $\gamma \geq 0$. It follows from our
main theorem that, in this case, we have convergence of $T_\ve$ to the
effective temperature $T_0$. More quantitatively:
\begin{cor}
	\label{2dcor}
	Let $d = 2$ and let $B$ be a non-vanishing divergence-free vector
	field. Then 
\be
\| T_\ve-T_0\|_{H^1(D)} \leq C \ve,
\ee
 where $T_0 = \Theta(\psi)$ where $\Theta$ is given by \eqref{limitTheta}.
\end{cor}

 In three dimensions, an important example of fibered fields are the
	smooth solutions of the magnetohydrostatic equations
 \begin{equation}
  (\curl B) \times B = \nabla p,
	\qquad \div B = 0,\qquad \text{ in } 	D \subset \mathbb{R}^3,
  \label{mhd}
 \end{equation}
	having the property that the pressure satisfies $\nabla p \not= 0$. As noted by Arnold   \cite{Arnold, AK}   since $|\nabla p|$ is nonvanishing by assumption,
	each surface $S_p$ is a smooth two-dimensional surface which admits two everywhere transverse non-vanishing tangent vector fields ($\curl B$ and $B$) and are thus two-dimensional tori or cylinders.  In this setting $B$ is fibered by its pressure, $\psi = p$.
	 It is straightforward to construct fields $B$ of this
	type which are axisymmetric, see e.g. \cite{Friedberg87} and
	\cite{Grad67}.
	It is an open problem (see \cite{Grad67}, \cite{Grad85})
	to construct such smooth magnetohydrostatic equilibria with $|\nabla p| > 0$
outside of Euclidean symmetry. 

	More generally, in three dimensions, given a  non-degenerate function $\psi:D\to \mathbb{R}$ whose level sets are tori along with functions $\theta, \phi :D \to \mathbb{R}$, any vector field of the form
	\begin{equation}
	 B = \nabla \psi \times \nabla \theta + \nabla \phi \times \nabla \chi
	 \label{Bformmix}
	\end{equation}
	is divergence-free. If $\chi$ is chosen so that $\chi = \chi(\psi, \phi)$,
	$B$ is fibered by $\psi$, and  known as ``integrable''
because the integral curves of $B$ obey a Hamiltonian system
with Hamiltonian $\chi$, and this Hamiltonian is integrable in the usual sense\footnote{
We suppose that with $B$ as in \eqref{Bformmix}, the functions
$\theta, \phi, \psi$ together form a coordinate system in $D$. Then for any
smooth
$u:D \to \mathbb{R}$,
we have $
 \nabla u = \pa_\psi u \nabla \psi + \pa_\phi u \nabla \phi + \pa_\theta u \nabla \theta
$
and so, writing
$
 J = \nabla \psi \times \nabla \theta \cdot \nabla \phi,$
which is nonvanishing by our assumption, we have the formula
\begin{equation}
 (B\cdot \nabla) u = \left[ \pa_\phi u + \iota(\psi, \theta, \phi) \pa_\theta u
 + \tau(\psi, \theta, \phi) \pa_\psi u\right]J,
 \label{Boperator}
\end{equation}
where $ \tau(\psi, \theta, \phi) := -\pa_\theta \chi(\psi, \theta, \phi)$ and where we have introduced the rotational transform $
 \iota(\psi, \theta, \phi) := \pa_\psi \chi(\psi, \theta, \phi).
$
There is a simple interpretation of the function $\chi$.
Consider any integral curve of $B$, parametrized
by $\phi$. That is, we consider $\Psi(\phi), \vartheta(\phi)$ defined by
\begin{align}
 \frac{\rmd }{\rmd \phi} \Psi &=  \frac{B\cdot \nabla\psi }{B\cdot \nabla \phi}
 = -\pa_\theta \chi,\label{ham1}\qquad \frac{\rmd}{\rmd \phi} \vartheta =\frac{B\cdot \nabla \theta}{B\cdot \nabla \phi}
 = \pa_\psi \chi,
\end{align}
with the understanding that the quantities on the right-hand sides are
evaluated at $(\psi, \theta, \phi) = (\Psi(\phi), \vartheta(\phi), \phi)$.
Thus the integral curves of $B$ satisfy a Hamiltonian system with
Hamiltonian $\chi$. Note that if $\pa_\theta \chi = 0$,
the above system is integrable (has a conserved quantity)
since $\psi$ is constant along the flow.
This also be seen from
the formula \eqref{Boperator}.}
when $\pa_\theta \chi = 0$
(see \eqref{ham1}). See Figure \ref{fig:chan}, right panel.
	Fields of this form play an important role in the problem of confining
	a plasma with a magnetic field \cite{Helandernote}.   Such fields may sometimes be regarded as MHS solutions held steady by external forcing (e.g. by current carrying coils in some particular geometry) \cite{CDG2,CDG3}. 
	
Suppose $\theta, \phi$ form a coordinate system
on $S_\psi$ and so we write $u = u(\psi, \theta, \phi)$.  Then if $B$ is as in
\eqref{Bformmix} with $\pa_\theta \chi = \pa_\phi \chi = 0$,
it follows
after writing $\iota(\psi) = \chi'(\psi)$,
\begin{equation}
(B\cdot \nabla) u = \left[ \pa_\phi u + \iota(\psi)\pa_\theta u\right] J,
 \label{introint}
\end{equation}
where $J = \nabla \psi \times \nabla \theta \cdot \nabla \phi$.
Generally, by a theorem of Sternberg \cite{sternberg}, if $B$
is \emph{any} nonvanishing divergence-free vector field fibered by a function $\psi$,
(in particular, this includes the case $\chi = \chi(\psi, \phi)$ of
\eqref{Bformmix})
then on each
$S_\psi$ there are coordinates $\theta, \phi$ and a number $\iota = \iota(\psi)$
so that, expressed in these coordinates, $B$ takes the form \eqref{introint}
for a function $J = J(\psi, \theta, \phi) > 0$. We call the function
$\iota$ from \eqref{introint} the rotational transform. Our main result in
three dimensions,  proven in \S \ref{3dsecmix},  is that provided $\iota$ is invertible with Lipschitz
inverse, we have convergence $T_\ve \to T_0$ in $H^1(D)$.

\begin{cor}
	\label{3dcor}
	Suppose that $B$ is a nonvanishing divergence-free vector field
	fibered by a function $\psi$. Suppose that the rotational transform
	$\iota$ from \eqref{introint} is invertible and 	for  some $L > 0$
	\begin{equation}
	 |\psi_1 - \psi_2| \leq L |\iota(\psi_1) - \iota(\psi_2)|
	 \label{mixlip}
	\end{equation}
	holds for all $\psi_1, \psi_2 \in I$.
	Then condition \eqref{measureergo} holds
	for any $\gamma > 2$ with $c = 1$.  Consequently,
	\be
\| T_\ve-T_0\|_{H^1(D)} \leq C \ve^{\frac{1}{3}},
\ee
 where $T_0 = \Theta(\psi)$ where $\Theta$ is given by \eqref{limitTheta}.
\end{cor}
In other words, we show that the ergodic condition holds
for integrable Arnol'd fibered fields $B$ with monotone rotational
transform. Such fields
are of specific interest in the plasma physics community, see the discussion in \cite{CDG2}. However such plasma equilibria, if they exist, may be unstable. Thus, it is important to also understand
the behavior of non-integrable fields $\pa_\theta \chi \not=0$.
There is an obstruction: the behavior of particle transport (and thus of heat) in
non-integrable fields can be quite complicated
because non-integrable Hamiltonian systems may
exhibit chaos.

In \cite{PHH21}, the authors consider non-integrable magnetic fields taking
the form \eqref{Bformmix} where
\begin{equation}
 \chi_\ve(\psi, \theta, \phi) = \chi_0(\psi) + \ve^a \chi_1(\psi, \theta, \phi),
 \label{perturbchi}
\end{equation}
and $a \geq 1/2$.
A modification of the proof of Theorem \ref{mainmixthm} (see Section \ref{3dnonintsec})
gives the following
generalization of Corollary \ref{3dcor} to fields of this type which are ``weakly nonintegrable.''  We require the following anisotropic Sobolev spaces tailored to the invariant tori: $f\in L^2(D)$ which are finite in the norm
\be\label{anissob}
\|f\|_{H^{(0, \gamma)}}^2 := \int_{\psi_-}^{\psi_+} \|f(\psi,\cdot)\|_{H^\gamma(S_\psi)}^2\, \rmd \psi .
\ee
\begin{theorem}
	\label{3dcornonint}
	Suppose that $B$ is as in \eqref{Bformmix} where $ \chi_\ve$ is given by \eqref{perturbchi} for $a\geq 1/2$  and satisfies $\|\pa_\theta  \chi_1\|_{L^\infty(D)}
	< 1$ and $\pa_\theta \chi_1|_{\pa D} = 0$. Moreover, assume that $\iota = \chi_0'$ satisfies
	the condition from Corollary \ref{3dcor}. Then letting $B_0$ be defined by \eqref{Bformmix}  with $\ve=0$, 
 	there is a constant $C = C(L)$ such that
\begin{align}
	   \|\nabla_{b_0}^\perp (T_\ve - T)\|_{L^2}
		 +\|T_\ve - T\|_{L^2}^2
		&\leq
		C \ve^{2a-1} \|\pa_\theta \chi_1\|_{L^2}^2
		+
		C\ve^{\frac{1}{3}}
		\left(\|\Delta T_0\|_{H^{(0, \gamma)}}^2 + \|\Delta T_0\|_{L^\infty}^2 \right),
		\label{3dnonintest}
		\\
		\|\nabla_{b_0} (T_\ve - T)\|_{L^2} &\leq
		C \ve^{2a} \|\pa_\theta \chi_1\|_{L^2}^2
		+
		 C\ve^{\frac{4}{3}}
		\left(\|\Delta T_0\|_{H^{(0, \gamma)}}^2 + \|\Delta T_0\|_{L^\infty}^2 \right).
		\label{3dnonintest2}
\end{align}
\end{theorem}

When $a > 1/2$, the estimate \eqref{3dnonintest} implies that $T_\ve \to T_0$
almost everywhere. When $a = 1/2$, we show the same result is true, and
in this case, even though $\res = \lim_{\ve \to 0} T_\ve - T_0 \in H^1(D)$
may not vanish almost everywhere,
we find $\res|_{S_\psi} = 0$ for all $\psi$ except possibly for a family
of $\psi$ lying in a set of measure zero. Here, $\res|_{S_\psi}$ denotes
the trace of the function $\res$ on the surface $S_\psi$, and this quantity is well-defined whenever $\res \in H^1(D)$, by the trace theorem.
In fact, the support of $\res = \lim_{\ve \to 0 } \rho_\ve$ is contained in the collection of non-ergodic surfaces
	$N(\gamma) = I \setminus E(\gamma)$, which has measure zero, where $E(\gamma)$ denote the family of ergodic surfaces
 	\begin{equation}
 	 E(\gamma) = \bigcup_{M > 0} E(\gamma, M).
	 \label{goodsets}
 	\end{equation}
	
\begin{cor}
	\label{3dcornonintspt}
 	Under the hypotheses of Corollary \ref{3dcornonint} with $a=1/2$, the sequence
	 $\rho_\ve = T_\ve - T_0$ converges weakly in $H^1$ to a distribution
	 $\res$ in $H^1(D)$ with the property that 
		\begin{equation}
		 \res|_{S\psi} = 0, \qquad \text{ whenever } \psi \in E(\gamma).
		 \label{weakvanishing}
		\end{equation}
		That is, the support of $\res$ is contained in $N(\gamma) = I\setminus E(\gamma)$.
\end{cor}

Our final result relates directly to the work of \cite{PHH21}. In that paper, the authors consider the sets
\begin{equation}
 \mathcal{N}(\ve) = \{ (\psi, \theta, \phi) : |\nabla_bT(\psi, \theta, \phi)|^2
 \geq \ve |\nabla_b^\perp T(\psi, \theta, \phi)|^2\},
 \label{nonintvol}
\end{equation}
where $b = B/|B|$ with $B$ as in \eqref{Bformmix}, and study them as a proxy for
the ``non-integrability'' of the field $B$.
Using the estimates from the above section, we can get an upper
bound on the measure set in the limit $\ve \to 0$.
\begin{prop}\label{propconj}
	Define $B$ as in \eqref{Bformmix} and $\mathcal{N}(\ve)$ as in \eqref{nonintvol}.
	Suppose that the boundary values $T_{\pm}$ from \eqref{effectivebc} satisfy
	$T_+\not= T_-$.
	Under the hypotheses of Corollary \ref{3dcornonint},
	there is a constant $C$ depending continuously on $\|T_0'\|_{L^\infty},$
	$\|1/T_0'\|_{L^\infty}$,
	$\|T_0''\|_{L^\infty}$, $(1 - \|\nabla \chi_1\|_{L^\infty})^{-1}$
	and $(2-\gamma)^{-1}$ so that
	\begin{equation}
	 \mu(\mathcal{N}(\ve)) \leq C  \left(\ve^{2a - 1}\|\pa_\theta \chi_1\|_{L^\infty}^2
	 + \ve^{1/3}\|\Delta T_0\|_{H^{(0, \gamma)}}^2\right).
	 \label{nonintmeasureestimate}
	\end{equation}
\end{prop}
Note that the ``integrable'' case corresponds to taking $\chi_1= 0$ and it follows that in this case
\begin{equation}
 \lim_{\ve \to 0} \mu(\mathcal{N}(\ve)) = 0.
 \label{}
\end{equation}
Since now $\chi = \chi_0$ is integrable, this agrees with the fact that
the effective volume of non-integrability is zero.
If $\chi_1$ is nonzero, we get the same result with $a > 1/2$
but if $a = 1/2$ we instead have
\begin{equation}
 \lim_{\ve\to 0} \mu(\mathcal{N}(\ve)) \leq
 C\|\pa_\theta \chi_1\|_{L^\infty}^2.
 \label{}
\end{equation}
This exhibits a relationship between the volume of the set \eqref{nonintvol}
and the non-integrability of the Hamiltonian $\chi$, captured by the $\theta$-dependence
of the perturbation $\chi_1$.
Our proof of the Proposition (see Section \ref{3dnonintsec}) partially confirms a conjecture announced in \cite{PHH21}.

\section{Proof of Theorem \ref{mainmixthm}}\label{secmainmixthm}
The limiting profile $\Theta(\psi)$ is a function of the flux function $\psi$ only, and is determined from the following heuristic. If $B$ is fibered, expanding $T_\ve = T_0 + \ve T_1$
	leads to
	\begin{align} \label{heurt0}
	 \div(b \nabla_b T_0) &=0,\\
	 \div(\nabla_b^\perp T_0) &= -\div(b \nabla_b T_1).
	 \label{heurt1}
	\end{align}
	Equation \eqref{heurt0} is underdetermined for $T_0$; indeed noting that $b\cdot \nabla \psi = 0$ one sees that any function
	\be
	T_0 = \Theta(\psi)
	\ee
	will satisfy \eqref{heurt0}.  This arbitrariness of $\Theta$ may be eliminated by considering the second condition. Indeed, note that
	since $b$ is tangent to each surface $S_\psi$, we have
	\begin{equation}
		\frac{\div(b \nabla_b T_1)}{|\nabla \psi|} \Big|_{S_\psi}  = \div_{S_\psi}\left(\left[ \frac{ b \nabla_b T_1}{|\nabla \psi|}\right]\Big|_{S_\psi}\right)
	\end{equation}
where
	$\div_{S_\psi}$ denotes the divergence operator on $S_\psi$. See Lemma \ref{lemdiv}. In light of this, equation
	\eqref{heurt1} comes with the following compatibility condition: on each invariant torus $S_\psi$,
	\begin{equation}
		\int_{S_\psi}  \frac{ \Delta T_0}{|\nabla \psi|}\rmd \mathscr{H}^{(d-1)} = 0,
	 \label{effectiveabstract}
	\end{equation}
	where $\rmd \mathscr{H}^{(d-1)}$ is the $(d-1)$ dimensional Hausdorff measure on $S_\psi$. Here we used that $\nabla_b^\perp T_0 = \nabla T_0$ since $b\cdot \nabla T_0 = 0$.
	Because  $T_0$ is constant on $S_{\psi}$, it follows from Lemma \eqref{streamderiv}
	(see \eqref{coareaderiv2}) that the solvability requirement
  \eqref{effectiveabstract} and the boundary conditions of \eqref{ADE} are satisfied if
	$\Theta$ is the unique solution to \eqref{effectivebc}. From \eqref{effectivebc}, we deduce that the effective temperature
distribution is give explicitly by
\be\label{limitTheta}
 \Theta(\psi) =T_-+ (T_+ -T_-)\frac{H(\psi;\ \psi_-)}{H(\psi_+;\psi_-) }  , \qquad \text{where} \qquad  H(\psi;\psi_-) := \int_{\psi_-}^\psi \frac{\rmd s }{\Gamma(s)} .
\ee

Note that in two-dimensions, $\Gamma(\psi)$ is simply the circulation of the vector field $B$ on the circle $S_\psi$:
\be
 \Gamma(\psi) =\int_{S_\psi} |\nabla \psi|\, \rmd \mathscr{H} = \int_{S_\psi} B\cdot \rmd \ell \qquad \text{in two dimensions}.
\ee

Observing \eqref{effectiveabstract}, we start with the following simple lemma.  We first introduce the homogenous fractional Sobolev seminorms on $S_\psi$. For each $S_\psi$ we pick coordinates $\theta_1,..., \theta_{d-1}$
on $S_\psi$ such that $\theta_j$ maps $S_\psi$ to $[0, 2\pi]$.
For $k \in \mathbb{Z}^{d-1}$, we define
\begin{equation}
 \widehat{u}(k) = \frac{1}{(2\pi)^{d-1}} \int_{[0,2\pi]^{d-1}}
 \prod_{j=1}^{d-1}   e^{ik_{j}\theta_{j}}
 u(\theta_1,\cdots, \theta_{d-1}) \rmd \theta_1\cdots \rmd \theta_{d-1},
 \label{fseries}
\end{equation}
and then for $\gamma \in \mathbb{R}$,  we define
$\|\cdot\|_{\dot{H}^\gamma}$
by
\begin{equation}
 \|u\|_{\dot{H}^\gamma(S_\psi)}^2 = \sum_{k \in \mathbb{Z}^{d-1}
 \setminus 0} |k|^{2\gamma} |\widehat{u}(k)|^2.
 \label{homoSob}
\end{equation}
Recall also that \eqref{anissob} defines the anisotropic spaces $H^{(0, \gamma)}$ tailored to the tori. The result is then
\begin{lemma}
	\label{productest}
 Suppose that $F \in H^{(0, \gamma)}(D)\cap L^\infty(D)$ for some $\gamma \geq 0$
 and that $F$ satisfies
 \begin{equation}
  \int_{S_\psi} \frac{F}{|\nabla \psi|} \, \rmd \mathscr{H}^{(d-1)}  = 0,
  \label{meanfree}
 \end{equation}
 for all $\psi \in [\psi_-, \psi_+]$. There is a constant $C$ depending
 only on $D$ so that
 for any $M > 0$ we have
 \begin{equation}
  \left| \int_D F u \, \rmd \mu  \right|
	\leq C M\|F\|_{H^{(0,\gamma)}}\|\nabla_b u\|_{L^2}
	+  C \mu(N(\gamma, M))^{1/2} \|F\|_{L^\infty} \|u\|_{L^2}.
  \label{productinftymix}
 \end{equation}

\end{lemma}

\begin{proof}
	By the co-area formula \eqref{coarea}, we have
	\begin{equation}
	 \int_D F u\, \rmd \mu   =
	 \int_{\psi_-}^{\psi_+} \int_{S_\psi} \frac{F}{|\nabla \psi|} u\, \rmd \mathscr{H}^{(d-1)}  \rmd \psi .
	 \label{}
	\end{equation}
	Now, for each $\psi$, we write
		\begin{equation}
		 \int_{S_\psi} \frac{F}{|\nabla \psi|}
		 u \, \rmd \mathscr{H}^{(d-1)}
		 = \int_{[0,2\pi]^{d-1} } G(\psi, \theta_1,..., \theta_{d-1})
		 u(\psi, \theta_1,..., \theta_{d-1}) \rmd \theta_1\cdots \rmd \theta_{d-1},
		 \label{}
		\end{equation}
		where $G =  \frac{F}{|\nabla \psi|} |h|^{1/2}  $,
		where we are writing the metric on $S_\psi$ as $h = h_{\alpha\beta} \rmd \theta^\alpha \rmd \theta^\beta$
		and $|h| = \det h_{\alpha\beta}$. By Parseval's theorem,
		\begin{equation}
		 \int_{[0,2\pi]^{d-1} } G(\psi, \theta_1,..., \theta_{d-1})
		 u(\psi, \theta_1,..., \theta_{d-1}) \rmd \theta_1\cdots \rmd \theta_{d-1}
		 = (2\pi)^{d-1}\sum_{k \in \mathbb{Z}^{d-1}} \widehat{G}(\psi, k) \widehat{u}(\psi, -k),
		 \label{}
		\end{equation}
		where $\widehat{F}, \widehat{u}$ are defined as in \eqref{fseries}. Now we note that
		\begin{equation}
		 \widehat{G}(\psi, 0)
		 = \int_{[0,2\pi]^{d-1}}  \frac{F}{|\nabla \psi|}|h|^{1/2}
		 \rmd \theta_1\cdots \rmd \theta_{d-1}
		 = \int_{S_\psi} \frac{F}{|\nabla \psi|}\, \rmd \mathscr{H}^{(d-1)}  = 0,
		 \label{}
		\end{equation}
		by assumption. We therefore have
		\begin{equation}
		 \int_{S_\psi} \frac{F}{|\nabla \psi|}
		 u\, \rmd \mathscr{H}^{(d-1)}
		 = (2\pi)^{d-1}\sum_{k \in \mathbb{Z}^{d-1}\setminus 0}
		 \widehat{G}(\psi, k) \widehat{u}(\psi, -k),
		 \label{}
		\end{equation}
		and it follows that
		\begin{equation}
		 \left| \int_D F u\, \rmd \mu  \right|
		 \leq C\int_{\psi_-}^{\psi_+}
		 \sum_{|k|\not=0} |\widehat{G}(\psi, k)\widehat{u}(\psi, -k)|\, \rmd \psi .
		 \label{}
		\end{equation}
		Now we split $[\psi_-, \psi_+] = E(\gamma, M)\cup N(\gamma, M)$ and bound
		\begin{multline}
			\int_{\psi \in E(\gamma, M)}
 		 \sum_{|k|\not=0} |\widehat{G}(\psi, k)\widehat{u}(\psi, -k)|\,\, \rmd \psi
		 \\
		 \leq
		 \int_{\psi \in E(\gamma, M)} \|G\|_{H^\gamma(S_\psi)} \|u\|_{\dot{H}^{-\gamma}(S_\psi)}\, \rmd \psi
		 \leq M \int_{\psi \in E(\gamma, M)}\|G\|_{H^\gamma(S_\psi)}
		 \|\nabla_B u\|_{L^2(S_\psi)}\, \rmd \psi \\
		 \leq M \|G\|_{H^{(0,\gamma)}(D)} \|\nabla_B u\|_{L^2(D)},
		 \label{}
		\end{multline}
		by the definition of $E(\gamma, M)$
		and
		\begin{multline}
 		 \int_{\psi \in N(\gamma, M)}
		\sum_{|k|\not=0} |\widehat{G}(\psi, k)\widehat{u}(\psi, -k)|\,\, \rmd \psi
		\leq \|G\|_{L^2(N(\gamma, M))}\|u\|_{L^2(D)}\\
		\leq \mu(N(\gamma, M))^{1/2} \|G\|_{L^\infty(D)} \|u\|_{L^2(D)}.
		 \label{}
		\end{multline}
		This gives the result.
\end{proof}

Our main result, Theorem \ref{mainmixthm}, will be a direct result of the following
estimate.
\begin{prop}
	\label{rhoest}
	Define $T_0=\Theta(\psi)$ where $\Theta$ is given by \eqref{limitTheta} and let $\rho = T_\ve - T_0$.	Under the hypotheses of Theorem \ref{mainmixthm},
	there is a constant $C$ depending only on the domain $D$
	so that for each $\ve > 0$ and $M > 0$, we have
	\begin{equation}
	\|\nabla_b \rho\|_{L^2}^2 + \ve \|\nabla_b^\perp \rho\|_{L^2}^2
	+\ve \|\rho\|_{L^2}^2
	\leq C \left(M^2 \ve^2 + \ve \mu(N(\gamma, M))\right) \left(\|\Delta T_0\|_{H^{(0,\gamma)}}^2 + \|\Delta T_0\|_{L^\infty}^2\right)
	 \label{mainestimate}
	\end{equation}
\end{prop}

\begin{proof}
	The remainder $\rho$ satisfies
	\begin{align}
	 \div (b\nabla_b \rho) + \ve \div(\nabla_b^\perp \rho)&=
	 - \ve \Delta T_0, \qquad \text{ in } D,\label{Reqnmix}\\
	 \rho|_{S_{\pm}} &= 0,
	 \label{}
	\end{align}
	where we wrote $\Delta T_0 = \div(\nabla_b^\perp T_0)$ since $b\cdot \nabla T_0 = 0$.
	If we multiply \eqref{Reqnmix} by $\rho$ and integrate over $D$, using the co-area
	formula \eqref{coarea} we find
		\begin{equation}
		 \int_D\left( |\nabla_b \rho|^2 + \ve |\nabla_b^\perp \rho|^2\,\right) \rmd \mu
		 = \ve \int_{D} \Delta T_0 \rho\, \rmd \mu
		 = \ve \int_{\psi_-}^{\psi_+} \int_{S_\psi} \frac{\Delta T_0}{|\nabla \psi|}
		 \rho\, \rmd \mathscr{H}^{(d-1)}  \rmd \psi .
		 \label{errortermmix}
		\end{equation}
		Recall that we have defined $T_0$ so that $F = \Delta T_0$ satisfies the condition \eqref{meanfree}.
		We can therefore apply Lemma \ref{productest}, and by \eqref{productinftymix} we have
			\begin{align}\nonumber
			\ve \bigg|\int_{\psi_0}^{\psi_1}\int_{S_\psi} \frac{\Delta T_0}{|\nabla \psi|}  \rho\, &\rmd \mathscr{H}^{(d-1)}  \rmd \psi \bigg|
			\\ \nonumber
			&\leq  M \ve \|\Delta T_0\|_{H^{(0, \gamma)}} \|\nabla_b \rho\|_{L^2} + \ve \mu(N(\gamma, M))^{1/2} \|\Delta T_0\|_{L^\infty}\|\rho\|_{L^2}\\
			&\leq \frac{1}{2}  \left( M^2 \ve^2 \|\Delta T_0\|_{H^{(0, \gamma)}}^2 + \frac{1}{\delta}
			\ve \mu(N(\gamma, M)) \|\Delta T_0\|_{L^\infty}^2 \right)
			+ \frac{1}{2} \|\nabla_b \rho\|_{L^2}^2 + \frac{\delta}{2} \ve \|\rho\|_{L^2}^2
			 \label{keystep}
			\end{align}
			for any $\delta > 0$.
			Since $\rho|_{\pa D} = 0$, by Poincar\'e's inequality we have
			\begin{equation}
			 \|\rho\|_{L^2}^2 \leq C_P \|\nabla \rho\|_{L^2}^2,
			 \label{poin}
			\end{equation}
			where $C_P$ is the Poincar\'e constant for $D$. Taking $\delta$ so that
			$\delta C_P$ is sufficiently small, we see
			from \eqref{errortermmix} and \eqref{keystep}
			that there is a constant $C > 0$
			depending only on $C_P$ so that
			\begin{equation}
			 \|\nabla_b \rho\|_{L^2}^2 + \ve \|\nabla_b^\perp \rho\|_{L^2}^2
			 +\ve \|\rho\|_{L^2}^2
			 \leq C \left(M^2 \ve^2 + \ve \mu(N(\gamma, M))\right) \left(\|\Delta T_0\|_{H^{(0, \gamma)}}^2 + \|\Delta T_0\|_{L^\infty}^2\right),
			 \label{keystep2}
			\end{equation}
			after using \eqref{poin} again to bound $\ve \|\rho\|_{L^2}^2
			\leq C' \|\nabla_b \rho\|_{L^2}^2 + C' \ve \|\nabla_b^\perp \rho\|_{L^2}^2$
			for another constant $C'$.

\end{proof}

\begin{proof}[Proof of Theorem \ref{mainmixthm}]
 	If we take $M = \ve^{-\frac{1}{2+c}}$, then writing $N_\ve = N(\gamma, \ve^{-\frac{1}{2+c}})$,
	 \eqref{mainestimate} gives
	\begin{align}\label{bounds1}
	  \|\nabla_b^\perp \rho\|_{L^2}
		+\|\rho\|_{L^2}^2
		&\leq C\ve^{\frac{c}{2+c}}\left(1  + \ve^{-\frac{c}{2+c}}\mu(N_\ve)\right)
		\left(\|\Delta T_0\|_{H^{(0, \gamma)}}^2 + \|\Delta T_0\|_{L^\infty}^2 \right),\\ \label{bounds2}
		\|\nabla_b \rho\|_{L^2} &\leq
		 C\ve^{1 + \frac{c}{2+c}}\left(1  + \ve^{-\frac{c}{2+c}}\mu(N_\ve)\right)
		\left(\|\Delta T_0\|_{H^{(0, \gamma)}}^2 + \|\Delta T_0\|_{L^\infty}^2 \right).
\end{align}
By assumption, $\lim_{\ve \to 0} \ve^{-\frac{c}{2+c}} \mu(N_\ve) = 0$ and the result follows.
\end{proof}

\section{Proof of Corollary \ref{2dcor}: The 2d case}
\label{2dsecmix}
	If $|B| > 0$ in $D$,
	$B$ is fibered by its
	streamfunction $\psi$, $B = \nabla^\perp\psi$.
	We bound
	\begin{equation}
	 \|u\|_{\dot{H}^0(S_\psi)}^2
	 = \sum_{k \in \mathbb{Z} \setminus 0}
	 |\widehat{u}(k)|^2
	 \leq
	 \sum_{k \in \mathbb{Z}}
	 |k|^2 |\widehat{u}(k)|^2
	 \leq \frac{C}{\inf_{S_\psi} |B|^2} \|\nabla_B u\|_{L^2(S_\psi)}^2
	 \label{}
	\end{equation}
	for a constant $C > 0$. 
	Here  we have used
	that $B$ spans the tangent space to $S_\psi$ at each point. Therefore,
	$E(0, M) = D$ whenever $M \geq \frac{C}{\inf_{D} |B|}$, and so
	$N(0, M)$ is empty in this case. Thus
	\eqref{measureergo} holds for any $c \geq 0$ and the result follows.

\section{Proof of Corollary \ref{3dcor}:  The 3d integrable case}
\label{3dsecmix}

We first show that if $\gamma > 2$, under the hypotheses of
Corollary \ref{3dcor}, the condition \eqref{measureergo} holds
with $c = 1$. We start by relating this condition to the ``Diophantine''
condition.

Let $I = [\psi_-,\psi_+]$. Fix $\iota = \iota(\psi)$ with $\iota \in L^\infty(I)$.
Let $ |(m, n)| = \sqrt{m^2 + n^2}$. We define
\begin{equation}
D(\gamma, M) = \left\{\psi \in I : |m + \iota(\psi)n| \geq \frac{1}{M |(m, n)|^\gamma}\quad \text{for all} \quad (m,n)\in \mathbb{Z}^2\setminus \{0\}
\right\}.
\label{diophantine}
\end{equation}
If $\psi \in D(\gamma, M)$ for some $\gamma, M$, we will say that $\psi$ is
``Diophantine'' and that the surface $S_\psi$ is a ``Diophantine surface''.
With $m + \iota(\psi)n$ replaced by $\omega \cdot (m, n)$ for $\omega \in \mathbb{R}^2$,
these sets play a fundamental role in the proof of the celebrated KAM theorem
\cite{arnold}.

At least when $\iota(\psi) = \psi$, these sets are empty if $\gamma < 2$
(see \cite{treschev}),
but it turns out
that if $\iota$ is bi-Lipschitz and $\gamma > 2$ they have positive measure;
in fact the complement of $\cup_{M > 0 } D(\gamma,M)$ has zero measure,
as the next result shows.
\begin{lemma}
	\label{diolem}
	Let $\iota : I \to\mathbb{R}$ be an invertible function
	and suppose there is $L > 0$ so that
	\begin{equation}
	 \frac{1}{L} |\psi_1 - \psi_2| \leq |\iota(\psi_1) - \iota(\psi_2)|
	 \leq L |\psi_1 - \psi_1|.
	 \label{bilipscitz}
	\end{equation}
 Define $D(\gamma, M)$ as in \eqref{diophantine} and let $\mu$ denote the
 one-dimensional Lebesgue measure.
 If $\gamma > 2$ and $M > 0$, there is a constant $K$ depending only on
 $\mu(I)$, $L$, and $1/(\gamma-2)$ so that
 \begin{equation}
  \mu(I \setminus D(\gamma, M)) \leq K \left( \frac{1}{M^{\frac{1}{1+\gamma}}} + \frac{1}{M}\right).
  \label{measureestimate}
 \end{equation}
\end{lemma}
Taking $M \to \infty$ in \eqref{measureestimate} shows that the set of $\psi$
which fails the condition 
\eqref{diophantine} for \emph{all} $M > 0$ has zero measure.  Equivalently, the set of $\psi$ the condition in \eqref{diophantine}  for \emph{some} $M$ has full measure, though
in general the complement may be nonempty. 

\begin{proof}
	This result follows from a straightforward modification
	of the argument from e.g. \cite{treschev}. We include the details
	here for the convenience of the reader.
	For $(m, n) \in \mathbb{Z}^2 \setminus (0,0)$, we define
\begin{equation}
 \Pi_{(m, n)}(\gamma, M)
 = \left\{ \psi \in I : |m + \iota(\psi) n| < \frac{1}{M|(m,n)|^\gamma}\right\},
 \label{picond}
\end{equation}
so that
\begin{equation}
 I \setminus D(\gamma, M) \subseteq \bigcup_{(m, n) \not=(0,0)}
 \Pi_{(m, n)}(\gamma, M).
 \label{}
\end{equation}
If $n\not=0$, $\Pi_{(m, n)}(\gamma, M)$ is contained in the interval
$[\psi_1, \psi_2]\subset I$ where $\psi_{1}$ are such that $\iota(\psi_{1})
=  \frac{M}{n} \frac{1}{|(m, n)|^\gamma} + \frac{m}{n}$ and
$\iota(\psi_{2}) = -\iota(\psi_1)$
(such $\psi_1, \psi_2$ exist and are unique since by \eqref{bilipscitz} $\iota$ is invertible).  These are the maximal and minimal values of $\psi$ such that \eqref{picond} holds for a given $(m,n)$.
Therefore
\begin{equation}
	\mu(\Pi_{(m, n)}(\gamma, M))
	\leq |\psi_1 - \psi_2|
	\leq \frac{L }{2M |n| |(m, n)|^\gamma}.
 \label{}
\end{equation}
If $n = 0$, $\Pi_{(m, 0)}(\gamma, M) = I$
when $|m| < M^{-1/(1 + \gamma)}$ and it is empty otherwise. Therefore
\begin{multline}
 \sum_{(m, n)\not=(0,0)} \mu(\Pi_{(m, n)}(\gamma, M))
 \leq \sum_{m \not=0} \mu(\Pi_{(m,0)}(\gamma, M))
 + \sum_{ (m, n)\not=(0,0), n\not=0}
 \mu(\Pi_{(m,n)}(\gamma, M))\\
 \leq |I| \frac{1}{M^{1/(1+\gamma)}} + \frac{L}{M}
 + \frac{L}{2M}
 \sum_{|m|, |n| \geq 1} \frac{1}{|n| |(m, n)|^\gamma},
 \label{}
\end{multline}
since there are two terms in the second sum on the first line with $m = 0$.
The last term here is bounded by
\begin{multline}
 2 \sum_{m = 1}^\infty \sum_{n = 1}^\infty \frac{1}{n}
 \frac{1}{(m^2 + n^2)^{\gamma/2}}
 \leq 2 \sum_{m = 1}^\infty \int_1^\infty \frac{\rmd z}{z(m^2 + z^2)^{\gamma/2}}
 = 2 \sum_{m = 1}^\infty \frac{1}{m^{\gamma}} \int_{1/m}^\infty \frac{\rmd w}{w(1 + w^2)^{\gamma/2}}
 \\
 \leq 2\sum_{m = 1}^\infty \frac{1}{m^{\gamma - 1}}
 \int_0^\infty \frac{\rmd w}{(1+ w^2)^{\gamma/2}}
 \label{}
\end{multline}
which is finite if $\gamma > 2$. In this case we therefore have
\begin{equation}
 \mu\left(I \setminus D(\gamma, M)\right)
 \leq \mu\left({\bigcup}_{(m,n)\not=(0, 0)}
 \Pi_{m, n}(\gamma, M)\right)
 \leq C_1 \left(\frac{1}{M^{1/(1+\gamma)}} + \frac{1}{M}\right),
 \label{}
\end{equation}
for a constant $C_1$ depending only on $|I|$, $L$ and $1/(\gamma - 2)$,
which proves \eqref{nonintmeasureestimate}.
\end{proof}

The values of $\psi \in D(\gamma, M)$  are sometimes called ``strongly non-resonant'', and the surfaces $S_\psi$ with $\psi \in D(\gamma, M)$
are called ``non-resonant flux surfaces''.

\begin{proof}[Proof of Corollary \ref{3dcor}]
	We first consider a field of the form \eqref{Bformmix} when $\chi = \chi(\psi)$
	and where $\theta, \phi$ form a coordinate system on each $S_\psi$.
	With $\iota(\psi) = \chi'(\psi)$, in this setting we have
	\begin{equation}
	( B \cdot \nabla) u = J \left[ \pa_\phi  + \iota(\psi)\pa_\theta \right] u,
	J = \nabla \psi \times \nabla \theta \cdot \nabla \phi
	 \label{boperpf}
	\end{equation}
	where $ J \not= 0$ by assumption.
	We claim that the Diophantine surfaces are ergodic, in the sense that there
	is a constant $C$ with
\begin{equation}
 D(\gamma, C M) \subseteq E(\gamma, M)
 \label{contain}
\end{equation}
for any $M > 0$. It follows from this claim that $N(\gamma, M) \subseteq I \setminus D(\gamma, CM)$
and so the condition \eqref{measureergo} holds
for any $c < 1$,
since \eqref{measureestimate} then implies that
\begin{equation}
   \mu(N(\gamma, M)) \leq \frac{C}{M},
	 \qquad M \geq 1.
 \label{measureest}
\end{equation}

We now prove \eqref{contain}.
Whenever $\psi \in D(\gamma, M)$, for any smooth function $u:S_\psi \to \mathbb{R}$,
we have
\begin{equation}
	\|u \|_{\dot{H}^{-\gamma}(S_\psi)}^2
	= \sum_{(m, n) \in \mathbb{Z}^2\setminus 0}
 \frac{|\widehat{u}( m, n)|^2}{\left( m^2 + n^2\right)^\gamma}
	\leq M^2 \sum_{(m, n) \in \mathbb{Z}^2\setminus 0}
 |m + \iota(\psi) n|^2|\widehat{u}( m, n)|^2.
 \label{pointofgoodset}
\end{equation}
Now we note that by \eqref{boperpf},
\begin{equation}
 \widehat{\left(\nabla_{B/J} u\right)}(\psi, m, n)
 = -2\pi i(m + \iota(\psi) n) \widehat{u}(\psi, m, n).
 \label{fouriersln}
\end{equation}
It follows that if $\psi \in D(\gamma, M)$, there are constants $C_1, C_2$
so that
\begin{equation}
	\|u \|_{\dot{H}^{-\gamma}(S_\psi)}^2
	\leq C_1 M^2
	\| \nabla_{B/J} u \|_{L^2(S_\psi)}^2
	\leq C_2 M^2 \|\nabla_b u\|_{L^2(S_\psi)}^2,
 \label{}
\end{equation}
and this gives \eqref{contain}.

We now show how to get the same result for any non-vanishing divergence-free
fibered field provided the rotational transform satisfies the bound
\eqref{mixlip}. We first show that the rotational transform is well-defined in
this setting; that is, that we can find coordinates so that $B$
takes the form \eqref{introint}.

In light of \eqref{divrelation}, since $\div B = 0$, it follows that $\frac{1}{|\nabla \psi|}$
is an integral invariant of $B|_{S_{\psi}}$ (namely $U = \frac{1}{|\nabla \psi|}$
is a conserved density along the flow of $B|_{S_{\psi}}$ on $S_\psi$) and so by Sternberg's theorem (Theorem 1 of \cite{sternberg}), $B$ is orbitally conjugate to a linear operator on $S_\psi$. That is, there are coordinates
$(\theta, \phi)$ on $S_\psi$ so that in these coordinates,
there is a nonvanishing function $J = J(\theta, \phi)$ so that
on $S_\psi$, $B$ takes the form
\begin{equation}
 (B\cdot \nabla^T)u
 = J \left( \pa_\phi + \iota \pa_\theta \right) u, \qquad u \in C^\infty(S_\psi)
 \label{linsternberg}
\end{equation}
for a real number $\iota$ (compare with \eqref{Boperator}), where
$\nabla^T$ denotes the tangential gradient on $S_\psi$, given by
$\nabla^T u = (\nabla  - \frac{\nabla \psi}{|\nabla \psi|^2} \nabla\psi\cdot \nabla) u$
when $u$ is a function defined in a neighborhood of $S_\psi$.
Applying this theorem on each $S_\psi$ then gives $\iota = \iota(\psi)$.

The above proof goes through with a minor change, which is that we want
to replace the fractional Sobolev norm $\|u\|_{\dot{H}^{\gamma}}$,
which was defined relative to a fixed coordinate system (because the Fourier
coefficients $\widehat{u}(k)$ depend on the choice of coordinates
in \eqref{fseries}), with a fractional Sobolev norm $\|u\|_{\dot{\widetilde{H}}^{\gamma}}$
defined relative to the coordinates guaranteed by Sternberg's theorem.
This means we want to modify
the definition of $E(\gamma, M)$ and define
\begin{equation}
 \widetilde{E}(\gamma, M) =
 \left\{ \psi \in I : \| u\|_{\dot{\widetilde{H}}^{-\gamma}(S_{\psi})}
 \leq M \|\nabla_B u\|_{L^2(S_{\psi})}, \text{ for all } u \in H^1(S_{\psi})\right\}.
 \label{mixingsetstern}
\end{equation}
It is clear that the proof of Theorem \ref{mainmixthm} goes through without change
if we replace the sets $E(\gamma, M)$ with $\widetilde{E}(\gamma, M)$.
It is also clear from the formula \eqref{linsternberg} and the above
argument that there is a constant $C > 0$ so that
\eqref{contain} holds with $E$ replaced with $\widetilde{E}$, if the rotational transform
$\iota = \iota(\psi)$ from \eqref{linsternberg} satisfies \eqref{mixlip}.

\end{proof}

\section{Proof of  Theorem \ref{3dcornonint}, Corollary \ref{3dcornonintspt} and Proposition \ref{propconj}:  the non-integrable case}
\label{3dnonintsec}
We now consider the non-integrable case $\chi_1\not=0$ of \eqref{perturbchi}.
With $B$ as in
\eqref{Bformmix}-\eqref{perturbchi}, we set
\begin{equation}
B_0 = \nabla \psi \times \nabla \theta + \iota(\psi) \nabla \phi \times \nabla \psi,
\qquad
B_1 =   \nabla \phi \times \nabla \chi_1
\label{}
\end{equation}
where $\iota(\psi) = \chi_0'(\psi)$. We then write $b = B/|B|$ in the form
\begin{equation}
b = b_0 + \ve^a b_1, \qquad
b_0 = \frac{B_0}{|B|}, \quad b_1 = \frac{B_1}{|B|}.
\label{}
\end{equation}

\subsection{Proof of Theorem \ref{3dcornonint}.}
We start by recording a simple estimate.
\begin{lemma}
 If $\|\nabla \chi_1\|_{L^\infty} < 1$ and
 $a \geq 1/2$,
 there is a constant $C > 0$ so that for any function $u \in H^1(D)$, we have
 \begin{equation}
	\frac{1}{C}\left(\|\nabla_{b} u\|_{L^2}^2 + \ve \|\nabla_{b}^\perp u\|_{L^2}^2\right)
	\leq \|\nabla_{b_0} u\|_{L^2}^2 + \ve \|\nabla_{b_0}^\perp u\|_{L^2}^2
	\leq C \left(\|\nabla_{b} u\|_{L^2}^2 + \ve \|\nabla_{b}^\perp u\|_{L^2}^2\right).
	\label{b0tob}
 \end{equation}
\end{lemma}
	\begin{proof}

	This follows after writing $\nabla_{b_0} = \nabla_{b} - \ve^a \nabla_{b_1}$,
	noting that $|b_1| \leq C|\nabla \chi_1|$,
	and
	\begin{equation}
	 |\nabla_{b_0} u| \leq |\nabla_{b} u| + \ve^a |\nabla \chi_1| |\nabla u|,
	 \qquad
	 |\nabla_{b_0}^\perp u|
	 = |(\nabla - b_0 \nabla_{b_0})u|
	 \leq |\nabla_b^\perp u|
	 + \ve^a |\nabla \chi_1| |\nabla u|
	 \label{}
	\end{equation}
	for smooth $u$. It follows that for $u \in H^1(D)$,
	\begin{equation}
	 \|\nabla_{b_0} u\|_{L^2}^2 + \ve \|\nabla_{b_0}^\perp u\|_{L^2}^2
	 \leq \|\nabla_b u\|_{L^2}^2 + \ve \|\nabla_b^\perp u\|_{L^2}^2
	 + \ve^{2a}\|\nabla \chi_1\|_{L^\infty}^2 \|\nabla u\|_{L^2}^2.
	 \label{}
	\end{equation}
	Provided $2a \geq 1$, this gives
	\begin{equation}
	 (1 - \ve \|\nabla \chi_1\|_{L^\infty})\|\nabla_{b_0} u\|_{L^2}^2 + \ve(1 - \|\nabla \chi_1\|_{L^\infty}^2) \|\nabla_{b_0}^\perp u\|_{L^2}^2
	 \leq \|\nabla_b u\|_{L^2}^2 + \ve \|\nabla_b^\perp u\|_{L^2}^2,
	 \label{}
	\end{equation}
	which is the second bound in \eqref{b0tob}. The first bound is nearly
	identical.
	\end{proof}

\begin{proof}[Proof of Corollary \ref{3dcornonint}]

Since $\nabla_{b_0} T_0 = 0$ we have
\begin{multline}
 \div(b \nabla_b T_0) + \ve \div(\nabla_b^\perp T_0)
 = \ve^a \div(b \nabla_{b_1} T_0) + \ve \Delta T_0
 -\ve \div( b \nabla_b T_0)\\
 =
 (\ve^a - \ve^{1+a}) \div(b \nabla_{b_1} T_0)
 + \ve \div (\nabla T_0).
 \label{}
\end{multline}
Then $\rho = T_\ve - T_0$ satisfies
\begin{align}
 \div(b \nabla_b \rho) + \ve \div(\nabla_b^\perp \rho)
 &=
 (\ve^{1+a} - \ve^{a}) \div(b \nabla_{b_1} T_0)
 - \ve \div (\nabla T_0)\qquad \text{ in } D,\label{Reqnmixnonint}\\
 \rho|_{S_{\pm}} &= 0.
 \label{}
\end{align}
Since $\pa_\theta \chi_1|_{\pa D} = 0$ by assumption, $b$ is tangent to
$\pa D$ and so if we multiply this by $\rho$ and integrate over $D$, we find
\begin{equation}
 \int_D |\nabla_b \rho|^2 + \ve |\nabla_{b}^\perp \rho|^2\, \rmd \mu
 = \ve \int_D \Delta T_0 \rho \, \rmd \mu
 + (\ve^a - \ve^{1+a})
 \int_D \nabla_{b_1}T_0 \nabla_b \rho\, \rmd \mu  ,
 \label{}
\end{equation}
after integrating by parts in the second term on the right-hand side.
Since $|\nabla_{b_1} T_0| \leq |\pa_\theta \chi_1| |T_0'|$,
for any $\delta > 0$ we have
\begin{equation}
 \|\nabla_{b} \rho\|_{L^2}^2 + \ve \|\nabla_{b}^\perp \rho\|_{L^2}^2
 \leq  \frac{\ve^{2a}}{2\delta}\|\pa_\theta \chi_1\|_{L^2}^2 \|T_0'\|_{L^\infty}^2
 + \frac{\delta}{2} \|\nabla_{b}\rho\|_{L^2}^2
 + \ve \left| \int_{\psi_-}^{\psi_+} \int_{S_\psi} \frac{\div \nabla T_0}{|\nabla \psi|}
  \rho\, \rmd \mathscr{H}^{(d-1)}  \rmd \psi \right|.
 \label{}
\end{equation}
Arguing as in \eqref{keystep}-\eqref{keystep2}
and using \eqref{b0tob}, this implies that there is a constant $C > 0$ so that
for any $M > 0$,
\begin{multline}
 \|\nabla_{b_0} \rho\|_{L^2}^2 + \ve \|\nabla_{b_0}^\perp\rho\|_{L^2}^2
 + \ve \|\rho\|_{L^2}^2\\
 \leq C \ve^{2a}\|\pa_\theta \chi_1\|_{L^2}^2 \|T_0'\|_{L^2}^2
 + C\left(M^2 \ve^2 + \ve \mu(N(\gamma, M))\right) \left(\|\Delta T_0\|_{H^{(0, \gamma)}}^2 + \|\Delta T_0\|_{L^\infty}^2\right).
 \label{}
\end{multline}
If we take $M = \ve^{-1/3}$ and use the estimate \eqref{measureest}
for $\mu(N(\gamma, M))$, we find
	\begin{align}
	   \|\nabla_{b_0}^\perp \rho\|_{L^2}
		 +\|\rho\|_{L^2}^2
		&\leq
		C \ve^{2a-1} \|\pa_\theta \chi_1\|_{L^2}^2
		+
		C\ve^{\frac{1}{3}}
		\left(\|\Delta T_0\|_{H^{(0, \gamma)}}^2 + \|\Delta T_0\|_{L^\infty}^2 \right),\\
		\|\nabla_{b_0} \rho\|_{L^2} &\leq
		C \ve^{2a} \|\pa_\theta \chi_1\|_{L^2}^2
		+
		 C\ve^{\frac{4}{3}}
		\left(\|\Delta T_0\|_{H^{(0, \gamma)}}^2 + \|\Delta T_0\|_{L^\infty}^2 \right).
\end{align}
\end{proof}

\subsection{Proof of Corollary \ref{3dcornonintspt}.}

Before proving Corollary \ref{3dcornonintspt}, we collect some preliminary results.
First, from the uniform bounds \eqref{3dnonintest} and \eqref{3dnonintest2},
it follows that that the sequence $T_\ve - T_0$ has weak limit
$\res$ in $H^1$ and that $\nabla_b (T_\ve - T_0)$ converges strongly to
$0$ in $L^2$.

If we knew that $\nabla_B \res$ was smooth, it would follow that
$\nabla_B \res = 0$ everywhere. We however only know that $\nabla_B \res$
is in $L^2$ and in particular the restriction $\nabla_B \res|_{S_\psi}$
need not be defined. The following result shows that $\nabla_B \res|_{S_\psi} = 0$
in a weak sense.
\begin{lemma}
 Let $\res = \lim_{\ve \to 0} T_\ve - T_0$. For any
 $\psi$ and any $v \in C^2(S_\psi)$, with $B' = B|\nabla \psi|^{-1}$,
 \begin{equation}
   \int_{S_\psi} \res \div_{S_\psi}(B' v)\, \rmd \mathscr{H}^{(d-1)}  = 0.
  \label{weaktrace}
 \end{equation}
\end{lemma}
\begin{remark}
 The statement of \eqref{weaktrace} holds with $B'$ replaced by $B$
 (or indeed any multiple of $B$) but it is more convenient for our purposes
 to use $B'$ since $\div_{S_\psi}(B' v) = |\nabla \psi| \div(B v) = |\nabla \psi| \nabla_Bv$,
 by \eqref{divrelation} and the fact that $\div B = 0$. See   Lemma \ref{MDElem}.
\end{remark}
\begin{proof}
	The idea is to use Stokes' theorem to write \eqref{weaktrace} in
	terms of an integral in the interior of $D$ and to integrate by parts
	in the interior to exploit the fact that $\nabla_B \res =0$ almost everywhere.

 Fix a surface $S_{\psi'}$.
 Define a cutoff function
$Q \in C^\infty(\mathbb{R})$ with $Q(z) = 1$ when $|z| \leq 1$ and $Q(z) = 0$
when $|z| > 2$. Define $Q_\delta(\psi) = Q( (\psi - \psi')/\delta)$, so that
$Q_\delta(\psi)$ vanishes when $|\psi - \psi'| < 2\delta$ and
$Q_\delta(\psi) \equiv 1$ when $|\psi - \psi'| \leq \delta$. Let
$D_-(\psi') = D \cap\{\psi \leq \psi'\}$. Then the boundary of $D_-(\psi')$
is $S_{\psi'} \cup S_{\psi_-}$, and if $\delta$ is sufficiently small, $Q_\delta$
vanishes identically on $S_{\psi_-}$.

Let $V$ denote the constant extension of $v$ to $D_-(\psi')$,
$V(\psi, \theta, \phi) = v(\psi', \theta, \phi)$ for all $(\psi, \theta, \phi)
\in D_-(\psi')$.
By Stokes' theorem, since the outer unit normal to $S_{\psi'}$ is
$n =\frac{\nabla \psi}{|\nabla \psi|}$, we have
\begin{multline}
 \int_{S_{\psi'}} \res \div_{S_{\psi'}}(B'u)\,  \rmd \mathscr{H}^{(d-1)} =
 \int_{S_{\psi'}} (n\cdot n)\, \res \div_{S_{\psi'}} (B' V) Q_\delta  \,  \rmd \mathscr{H}^{(d-1)}\\ =
 \int_{D_-} Q_\delta \div \left( \res \div_{S_{\psi'}} (B' V) \frac{\nabla \psi}{|\nabla \psi|}\right)
 \rmd \mu
 + \int_{D_-}  \res \div_{S_{\psi'}} (B' V)\, \frac{\nabla \psi}{|\nabla \psi|} \cdot (\nabla Q_\delta)\, \rmd \mu  .
 \label{trace1}
\end{multline}
The first term here is bounded by
\begin{equation}
 \| \div_{S_{\psi'}} (B' V)\|_{H^1( |\psi -\psi'| \leq \delta)} \|\res \|_{H^1(|\psi-\psi'|\leq \delta)}
 \leq C \delta^{1/2}\|v\|_{C^2(S_{\psi'})} \|\res\|_{H^1(D)}.
 \label{}
\end{equation}
As for the second term, we use the identity \eqref{divrelation} to write
\begin{equation}
 \div_{S_\psi}( B' V) = \frac{1}{|\nabla \psi|} \div(B V),
 \label{}
\end{equation}
and since $\frac{\nabla \psi}{|\nabla \psi|} \cdot (\nabla Q_\delta) = \frac{1}{\delta}|\nabla \psi| Q_\delta'$
and $b$ is tangent to $S_{\psi'}$,
the second term in \eqref{trace1} is
\begin{equation}
 \frac{1}{\delta} \int_{D_-}  \res \div_{S_{\psi'}} (B' V)\, |\nabla \psi| Q_\delta'\, \rmd \mu
  = \frac{1}{\delta}\int_{D_-} \res \div(B  V) Q_\delta'\, \rmd \mu
	 =-\frac{1}{\delta}\int_{D_-} \nabla_B (\res Q_{\delta}') V \, \rmd \mu  ,
 \label{}
\end{equation}
after integrating by parts.
We have therefore shown that for any $\delta > 0$,
\begin{multline}
 \left| \int_{S_{\psi'}} \res \div_{S_{\psi}} (B' V)\,  \rmd \mathscr{H}^{(d-1)}\right|
 \leq C \delta^{1/2}\|v\|_{C^2(S_{\psi'})} \|\res\|_{H^1(D)}
 + \frac{1}{\delta} C \left(\|\nabla_B \res\|_{L^2} + \|\nabla_B Q_\delta'\|_{L^2}\right) \|V\|_{L^2}
 \\
 = C \delta^{1/2}\|v\|_{C^2(S_{\psi'})} \|\res\|_{H^1(D)},
\end{multline}
where we used that $\nabla_B u = 0$ whenever $u = u(\psi)$ and that $\nabla_B \res = 0$
in $L^2$. Taking $\delta \to 0$ gives the claim.
\end{proof}
The condition \eqref{weaktrace} nearly says that $\res = 0$ on each $S_\psi$
in the sense of distributions, but in order to conclude this we would need to
know that any test function $v$ can be written in the form
$v = \div_{S_\psi} (B w)$ for some test function $w$. This need not be possible
on an arbitrary surface $S_{\psi}$, but by the next Lemma it is possible
provided $S_\psi$ is a Diophantine surface. We set
\begin{equation}
 D(\gamma) = \bigcup_{M > 0} D(\gamma, M).
 \label{fulldio}
\end{equation}
Note that by \eqref{contain}, $D(\gamma) \subseteq E(\gamma)$ where
$E(\gamma) = \cup_{M > 0} E(\gamma, M)$ denotes the collection of all ergodic
values of $\psi$. Note also that by Lemma \ref{diolem}, the complement
$I \setminus D(\gamma)$ has zero measure when $\gamma > 2$.
\begin{lemma}
	\label{MDElem}
	Fix $\gamma > 2$ and define $D(\gamma)$ as in \eqref{fulldio}.
	If $\psi \in D(\gamma)$ and $v \in H^{s+\gamma}(S_{\psi})$ for some
	$s \geq 0$, there is
	$w \in H^{s}(S_{\psi})$ satisfying
	\begin{equation}
	 \div_{S_\psi}(B'w) = v,
	 \qquad B' = B |\nabla \psi|^{-1}.
	 \label{MDE}
	\end{equation}
\end{lemma}
\begin{remark}
 The equation \eqref{MDE} is sometimes known as the ``magnetic
 differential equation''.   Lemma \ref{MDElem} simply says that you can solve this equation on good (sufficiently ergodic) flux surfaces.
\end{remark}
\begin{proof}
 We start by using \eqref{divrelation} to write
 \begin{equation}
  \div_{S_\psi}(B' w) = \frac{1}{|\nabla \psi|} \div(B w) = \frac{1}{|\nabla \psi|}
	\nabla_B w,
  \label{}
 \end{equation}
 and so \eqref{MDE} takes the form
 \begin{equation}
  \nabla_{B/J} w = V, \qquad V = J |\nabla \psi | v,
  \label{}
 \end{equation}
 where recall $J = |g|^{-1/2} =
 |(\nabla \psi \times \nabla \theta) \cdot \nabla \phi|$. We now define
 \begin{equation}
  \widehat{w}(m, n) = \frac{i}{2\pi} \frac{1}{m + \iota(\psi) n}
	\widehat{V}(m, n).
  \label{}
 \end{equation}
 Because $\psi \in D(\gamma)$, for some $M > 0$, we have the bound
 \begin{equation}
  |\widehat{w}(m, n)|^2  \leq M^2 \left( m^2 + n^2\right)^{\gamma} |\widehat{V}(m, n)|^2,
  \label{}
 \end{equation}
 so that in particular $w \in H^s(S_\psi)$ whenever $v = V/(J|\nabla \psi|)
 \in H^{s+\gamma}(S_\psi)$.  By the identity \eqref{fouriersln}, it follows that
 $w$ satisfies \eqref{MDE}.
\end{proof}

\begin{proof}[Proof of Corollary \ref{3dcornonintspt}]
	Fix $\gamma > 2$ and define $D(\gamma)$ as in \eqref{fulldio}.
	It follows from Lemma \ref{MDElem} that given $v \in C^\infty(S_\psi)$,
	there is $w \in C^\infty(S_\psi)$  so that
	with $\div_{S_\psi}(B |\nabla \psi| w) = v$. It then follows from
	Lemma \ref{weaktrace} that
	\begin{equation}
	 \int_{S_\psi} \res v \, \rmd \mathscr{H}^{(d-1)}
	 =
	 \int_{S_\psi} \res \div_{S_\psi}(B |\nabla \psi| w) \, \rmd \mathscr{H}^{(d-1)}
	 = 0,
	 \label{}
	\end{equation}
	as required.
\end{proof}

\subsection{Proof of Proposition \ref{propconj}: effective volume of non-integrability}

	It follows immediately from the equation \eqref{effectivebc}
	(or, equivalently, \eqref{limitTheta}) for $T_0 = \Theta(\psi)$ that either
	$T_0$ is constant or else $T_0'$ is nonvanishing.
	Since $T_0|_{S_+} = T_{+}\not=T_- = T_0|_{S_-}$ we
	have that $T_0'$ is nonvanishing and in particular $T_0'$ is bounded below.
  With $\lambda = \min_D \frac{1}{|T_0'|}$,
 for any set $N \subseteq D$ we therefore have
 \begin{equation}
  \mu(N) \leq \lambda^2 \int_N |T_0'|^2\, \rmd \mu
	\leq C\left(\lambda^2 \int_N |\nabla_b^\perp T|^2 \, \rmd \mu  + \lambda^2 \int_N |\nabla_b^\perp \rho|^2\, \rmd \mu  \right),
  \label{}
 \end{equation}
 where we used that $T_0' = |\nabla\psi|^{-2} \nabla T_0\cdot \nabla \psi$
 and that $|\nabla T_0| \leq |\nabla_b^\perp T| \leq |\nabla_b^\perp T|
 + |\nabla_b^\perp \rho|^2$.
  In particular, with $N = N(\ve)$,
	since $\|\nabla_b^\perp T\|_{L^2(N(\ve))}^2 \leq \ve^{-1}\|\nabla_b T\|_{L^2(N(\ve))}^2$,
 \begin{multline}
  \mu(N(\ve)) \leq C\lambda^2\left(  \ve^{-1} \int_{N(\ve)} |\nabla_b T|^2\, \rmd \mu  +
	 \int_{N(\ve)} |\nabla_b^\perp \rho|^2\, \rmd \mu  \right)\\
	\leq
	C\lambda^2 \left(
	\ve^{-1} \int_{N(\ve)} |\nabla_b T_0|^2\, \rmd \mu  +
	 \ve^{-1} \int_{D} |\nabla_b \rho|^2\, \rmd \mu  +
	 \int_{D} |\nabla_b^\perp \rho|^2\, \rmd \mu  \right)\\
	 \leq
	 C\lambda^2 \ve^{-1} \left(
 	 \int_{N(\ve)} |\nabla_b T_0|^2\, \rmd \mu  +
 	  \int_{D} |\nabla_b \rho|^2\, \rmd \mu  +
 	 \ve \int_{D} |\nabla_b^\perp \rho|^2\, \rmd \mu  \right).
  \label{boundforneps}
 \end{multline}
 Since $b = b_0 + \ve^{a} b_1$ and $\nabla_{b_0}T_0 = 0$, we have
 \begin{equation}
  \|\nabla_b T_0\|_{L^2(D)}^2 \leq \ve^{2a} \|\nabla_{b_1} T_0\|_{L^2(D)}^2
	\leq \ve^{2a} \|\pa_\theta \chi_1\|_{L^2(D)}^2\|T_0'\|_{L^\infty(D)}^2.
  \label{nablabt0}
 \end{equation}
 To handle the second and third terms in \eqref{boundforneps}, we
 use \eqref{3dnonintest}-\eqref{3dnonintest2} combined with \eqref{b0tob}. By \eqref{nablabt0}
 we therefore have
 \begin{equation}
  \mu(N(\ve))
	\leq
	C\lambda^2
	\left( \ve^{2a-1}\|\pa_\theta \chi_1\|_{L^2(D)}^2\|T_0'\|_{L^\infty(D)}^2 + \ve^{1/3}\|\Delta T_0\|_{H^{(0, \gamma)}(D)}^2\right).
\end{equation}

\appendix

\section{Geometric identities from the co-area formula}
\label{geosec}
In this section we collect some geometric formulas that we will use repeatedly.  In what follows, we fix $\psi:D\to \mathbb{R}$ such that $|\nabla \psi|\neq 0$ on $D$ and so that the level surfaces
 $S_\psi$ are codimension one manifolds which foliate $D$. Let $\psi_- =\inf_D \psi$ and $\psi_+= \sup_D \psi$.
We will use the co-area formula
\begin{equation}
 \int_{D} u\, \rmd \mu   = \int_{\psi_-}^{\psi_+}\int_{S_\psi} \frac{u}{|\nabla \psi|} \, \rmd \mathscr{H}^{(d-1)}{\rm d} \psi,
 \label{coarea}
\end{equation}
 see e.g. \cite{evansgariepy}.
We start with a simple result that generalizes Lemma E.2 from \cite{CDG1}.
\begin{lemma}
	\label{streamderiv}
	If $F \in H^2(D)$, we have
 \begin{equation}
  \frac{{\rm d}}{{\rm d} \psi} \int_{S_\psi} F \, \rmd \mathscr{H}^{(d-1)} =
	\int_{S_{\psi}} \div \left(\frac{\nabla \psi}{|\nabla \psi|} F\right)\,\rmd \mathscr{H}^{(d-1)}.
  \label{coareaderiv}
 \end{equation}
\end{lemma}
\begin{proof}
	We start with the observation that
	\begin{equation}
	 \int_{S_{\psi_2}} F \, \rmd \mathscr{H}^{(d-1)} - \int_{S_{\psi_1}} F\, \rmd \mathscr{H}^{(d-1)}
	 = \int_{D_{\psi_1, \psi_2}} \div\left( \frac{\nabla \psi}{|\nabla \psi|} F\right)\, \rmd \mu  ,
	 \label{bdyclaim}
	\end{equation}
	where $D_{\psi_1, \psi_2} = \cup_{\psi_1 \leq \psi' \leq \psi_2} S_{\psi'}$
	denotes the region bounded by the surfaces $S_{\psi_1}, S_{\psi_2}$. Indeed,
	by the divergence theorem,
	\begin{equation}
	 \int_{D_{\psi_1, \psi_2}} \div\left( \frac{\nabla \psi}{|\nabla \psi|} F\right)\, {\rm d}\mu
	 = \int_{S_{\psi_2}} n^{S_{\psi_2}}\cdot \frac{\nabla \psi}{|\nabla \psi|} F \, \rmd \mathscr{H}^{(d-1)}
	 + \int_{S_{\psi_1}} n^{S_{\psi_1}}\cdot \frac{\nabla \psi}{|\nabla \psi|} F \, \rmd \mathscr{H}^{(d-1)},
	 \label{bdyclaimdiv}
	\end{equation}
	where $n^{S_{\psi}}$ denotes the outward-pointing unit normal
	to $S_\psi$. Then
	$n^{S_{\psi_2}} = \frac{\nabla\psi}{|\nabla \psi|}\big|_{S_{\psi_2}}$
	and
	$n^{S_{\psi_1}} = -\frac{\nabla\psi}{|\nabla \psi|}\big|_{S_{\psi_1}}$, so
	\eqref{bdyclaimdiv} gives \eqref{bdyclaim}.
	Dividing \eqref{bdyclaim} by $\psi_2 - \psi_1$ and taking the limit
	gives \eqref{coareaderiv}.
\end{proof}
In particular, if $F = F(\psi)$ is constant on $S_\psi$,
writing
$\Delta F = \div(\nabla F) = \div (\nabla \psi F')$ we have
	\begin{equation}
	 \int_{S_\psi} \Delta F \rmd \mathscr{H}^{(d-1)}
	 = \frac{\rmd}{\rmd \psi } \left( \int_{S_{\psi}} |\nabla \psi| F' \rmd \mathscr{H}^{(d-1)}\right)
	 = \frac{\rmd}{\rmd \psi } \left( \left[\int_{S_\psi} |\nabla \psi| \rmd \mathscr{H}^{(d-1)}\right]
	 F'\right).
	 \label{coareaderiv2}
	\end{equation}

Another consequence of the formula \eqref{coarea} is the following
\begin{lemma}
\label{lemdiv} Let $X$ be a vector field defined
in $D$ with the property that $X|_{S_{\psi}}$ is tangent to
$S_\psi$.  Then the divergence $\div X$ in $D$ is related to the divergence
operator $\div_{S_\psi}$ on $S_\psi$ by
\begin{equation}
		\frac{\div X}{|\nabla \psi|} \Big|_{S_\psi}  = \div_{S_\psi}\left(\left[ \frac{ X}{|\nabla \psi|}\right]\Big|_{S_\psi}\right).
 \label{divrelation}
\end{equation}
In particular, if $X$ is divergence-free in $D$, then $\varrho:= \frac{1}{|\nabla \psi|}\big|_{S_\psi}$ is a density conserved by $X$ on $S_\psi$.
\end{lemma}
\begin{proof}
This can be seen by working in local coordinates but it is simpler to use \eqref{coarea}
and note that if $u \in C^\infty(D)$ is any test function then
\begin{align*}
 -\int_D \div X u\, \rmd \mu   &= \int_D (X\cdot \nabla) u\, \rmd \mu
 = \int_{\psi_-}^{\psi_+} \int_{S_\psi} (X\cdot \nabla) u\,\frac{ \rmd \mathscr{H}^{(d-1)}}{|\nabla \psi|}\\
 &= \int_{\psi_-}^{\psi_+} \int_{S_\psi} (X\cdot \nabla^T) u \frac{ \rmd \mathscr{H}^{(d-1)}}{|\nabla \psi|}
 = -\int_{\psi_-}^{\psi_+} \int_{S_\psi} \div_{S_\psi}\left(\frac{ X}{|\nabla \psi|}
 \right) u \, \rmd \mathscr{H}^{(d-1)},
 \label{}
\end{align*}
where we used that $X\cdot \nabla = X\cdot \nabla^T$ on $S_\psi$,
where $\nabla^T$ denotes the tangential gradient on $S_\psi$,
$\nabla^T U = (\nabla - \frac{\nabla \psi}{|\nabla \psi|^2} \nabla \psi\cdot \nabla)u$
whenever $u$ is an extension of $U$ from $S_\psi$ to a neighborhood of $S_\psi$.
By \eqref{coarea}, the left-hand side is
$-\int_{\psi_-}^{\psi_+} \int_{S_\psi} \frac{\div X}{|\nabla \psi|} u\, \rmd \mathscr{H}^{(d-1)}$. Then \eqref{divrelation} follows since $u$ is arbitrary.
\end{proof}

\subsection*{Acknowledgements.} We thank P. Constantin, P. Helander, S. Hudson, E. Paul, and P. Torres de Lizaur for numerous insightful discussions.
 The research of TD was partially supported by NSF-DMS grant 2106233. The work of DG was partially supported
 by the Simons Center for Hidden Symmetries and Fusion Energy
award \# 601960. The studies of HG are supported by Ford Foundation and
the NSF Graduate Research Fellowship under grant DGE-2039656.

\bibliographystyle{abbrv}

\end{document}